\documentclass[12pt]{amsart}

\title{Convexity inequalities for eigenvalues and log-concavity of eigenfunctions}
\author{Paul Bryan}
\address{School of Mathematical and Physical Sciences, Macquarie University, Sydney, NSW, Australia}
\email{paul.bryan@mq.edu.au}
\author{Julie Clutterbuck}
\address{School of Mathematics, Monash University, 9 Rainforest Walk, VIC 3800 Australia}
\email{Julie.Clutterbuck@monash.edu}
\author{Cale Rankin}
\address{School of Science, University of New South Wales Canberra}
\email{c.rankin@unsw.edu.au}
\date{}

\usepackage[a4paper,margin=1.5cm,includehead]{geometry}
\usepackage{parskip}

\usepackage{amsmath}
\usepackage{amssymb}
\usepackage{mathtools}
\usepackage{graphicx}
\usepackage[dvipsnames]{xcolor}
\usepackage[shortlabels]{enumitem}
\usepackage{multicol}
\usepackage{tikz}
\usepackage{xcolor}
\usepackage{environ}
\RequirePackage{amsthm}
\RequirePackage{hyperref}
\RequirePackage{cleveref}

\usepackage[normalem]{ulem}

\newtheorem{thm}{Theorem}[section]
\newtheorem{lem}[thm]{Lemma}

\newtheorem{cor}[thm]{Corollary}

\newenvironment{pf}{\begin{proof}}{\end{proof}}

\theoremstyle{definition}
\newtheorem{defn}[thm]{Definition}

\newtheorem{assume}[thm]{Assumption}

\theoremstyle{remark}
\newtheorem{rem}[thm]{Remark}

\crefname{lem}{Lemma}{Lemmata}
\crefname{prop}{Proposition}{Propositions}
\crefname{thm}{Theorem}{Theorems}
\crefname{cor}{Corollary}{Corollaries}
\crefname{defn}{Definition}{Definitions}
\crefname{example}{Example}{Examples}
\crefname{eg}{Example}{Examples}
\crefname{remark}{Remark}{Remarks}
\crefname{assume}{Assumption}{Assumptions}

\usepackage[backend=bibtex]{biblatex}
\begin{document}

\maketitle
\begin{abstract}
  We give simple new proofs of two well-known results for the Schr\"odinger operator: first, the Brunn--Minkowski inequality for Dirichlet eigenvalues and, second, the log-concavity of the first Dirichlet eigenfunction. Our proof of the first applies to a class of domains including $C^{1,1}$ connected domains and convex potentials. In the special case of convex domains, the second result is a simple corollary.
\end{abstract}

\section{Introduction}
\label{sec:org46972ca} Let \(\Omega_0, \Omega_1 \subseteq \mathbf{R}^n\) be bounded domains (connected,
open sets) and let $V : \overline{\Omega_0 \cup \Omega_1} \to \mathbb{R}$ be a potential. The first
Dirichlet eigenvalue for the Schr\"odinger operator $-\Delta + V$, which we denote by \(\lambda_i =
\lambda(\Omega_i)\), \(i=0,1\) is given by the Rayleigh quotient,
\[ \lambda_i = \inf \left\{\frac{\int_{\Omega_i} |\nabla u|^2 + Vu^2dx}{\int_{\Omega_i} u^2 dx} : u \in
W^{1,2}_{0} (\Omega_i \to \mathbf{R}) \backslash \{0\}\right\}.
\]

\begin{assume}
\label{regularity-assumption}

We assume that $\Omega_0, \Omega_1$ and $V$ are sufficiently regular so that the first eigenfunctions $\lambda_1, \lambda_2$ are $C^1$ up to the boundary. For example, domains with $C^{1,1}$ boundary combined with potentials $V \in C^\alpha(\overline{\Omega_0 \cup \Omega_1})$ satisfy this requirement \cite[Theorem 9.15]{MR1814364}. We also assume that $V$ extends to a convex function on the convex hull of $\Omega_0 \cup \Omega_1$
\end{assume}

For \(t \in [0, 1]\), the deformation of $\Omega_0$ into $\Omega_1$ by Minkowski sums is denoted by
\begin{equation}
  \label{eq:mink-sum-def} \Omega_t = (1-t)\Omega_0 + t \Omega_1 := \{ (1-t)x_0+tx_1 ; x \in \Omega_0
\text{ and } x_1 \in \Omega_1\}.
\end{equation}

Our goal in this paper is to give simple new proofs of the following theorems originally due to
Brascamp and Lieb \cite{zbMATH03522267}.

\begin{thm}[{\textbf{Brunn-Minkowski inequality for
eigenvalues} \cite[Theorem 6.2]{zbMATH03522267}}]
  \label{org5c5abe2} Let $\Omega_0,\Omega_1$ and $V$ satisfy Assumption \ref{regularity-assumption}. Let \(\lambda_t = \lambda(\Omega_t)\) be the
first Dirichlet eigenvalue of \(\Omega_t\). Then $\lambda_t \leq (1-t)\lambda_0 + t \lambda_1$.
\end{thm}

As an immediate corollary of our proof we obtain:

\begin{thm}[{\textbf{Log-concavity of first eigenfunction} \cite[Theorem 6.1]{zbMATH03522267}}]
  \label{thm:log-concave} Let $\Omega$ be a convex domain and $V:\Omega \rightarrow \mathbf{R}$ be
convex. Then the first Dirichlet eigenfunction is log-concave.
\end{thm}

Our proofs rely on
Colesanti's \cite{zbMATH02181983} idea of using the sup-convolution. Recall for functions \(u_i : \Omega_i \to \mathbf{R}\), the sup-convolution defined for \(z \in \Omega_t\) is
\begin{equation}
  \label{eq:sup-conv-def} \overline{u}(z) =
\sup\{u_0^{1-t}(x_0)u_1^t(x_1) : z = (1-t)x_0 + tx_1, \> x_i \in \Omega_i\}.
\end{equation}

However, where Colesanti's proof requires convexity of domains and log-concavity
of eigenfunctions, our proof works for connected domains with first eigenfunction $C^1$ up to the boundary, and
\textit{establishes} the log-concavity of eigenfunctions as a simple
corollary. On the other hand, Colesanti's result established, for the first
time, conditions for equality --- something not captured by our proof. Moreover, our results easily recover Theorem \ref{org5c5abe2} for convex $\Omega_i$ via approximation since if a sequence of convex domains $\Omega_k$ converges
to a convex domain $\Omega$ in Hausdorff distance then $\lambda(\Omega_k) \rightarrow \lambda(\Omega)$. Recently, Brunn--Minkowski inequalities for eigenvalues of elliptic operators have been considered by a number of authors \cite{MR2141854,MR2673741,MR4024953} and interesting new proofs have been provided for some of Brascamp \& Lieb's other results   \cite{MR4958493}.

Our plan is to prove  Theorems \ref{org5c5abe2} and \ref{thm:log-concave} in Section \ref{sec:org443bc0d} taking for granted some general regularity properties of sup-convolutions proved in Section \ref{sec:org184251d}. This allows the reader to  focus on the main idea of the proof.

\section{Proof of Theorems \ref{org5c5abe2} and \ref{thm:log-concave}}
\label{sec:org443bc0d}

\begin{pf}[Proof. (Theorem \ref{org5c5abe2})] For $i=0,1$ let \(u_i : \Omega_i \to \mathbf{R}\)
denote the positive first eigenfunctions so that \(-\Delta u_i + V u_i = \lambda_i u_i\) with $u_i|_{\partial \Omega_i} = 0$ . Then the sup-convolution \(\overline{u} : \Omega_t \to \mathbf{R}\)  given by \eqref{eq:sup-conv-def} is locally semiconvex and globally Lipschitz on $\Omega_t$ (Lemmas \ref{org15f8af8}, \ref{org7214da8}).

Fix $z \in \Omega_t$. The supremum  \eqref{eq:sup-conv-def} is attained at $(x_0,x_1) \in \Omega_0 \times \Omega_1$ (Lemma \ref{org2260c5a}).  Thus
the function $h(\xi) := u_0(x_0 + t\xi)^{1-t}u_1(x_1-(1-t)\xi)^t$ has a maximum at $\xi
= 0$, whereby
  \begin{equation}
    \label{eq:first-deriv} \frac{\nabla u_0}{u_0}(x_0) = \frac{\nabla u_1}{u_1}(x_1).
  \end{equation} By semiconvexity, for Lebesgue a.e. $z \in \Omega_t$, $\bar{u}$ has
Alexandrov second derivatives satisfying
\[ D_{ii}\overline{u}(z) = \lim_{\tau \rightarrow 0}\frac{\overline{u}(z+\tau e_i) +
\overline{u}(z-\tau e_i) - 2\overline{u}(z)}{\tau^2}.
\] Using that the supremum defining $\bar{u}(z)$ is attained at $(x_0,x_1) \in \Omega_0
\times \Omega_1$, we have
  \begin{align*} D_{ii}\overline{u}(z) &\geq \lim_{\tau \rightarrow 0} \tau^{-2}\Big[ u_0(x_0+\tau
e_i)^{1-t}u_1(x_1+\tau e_i)^t + u_0(x_0-\tau e_i)^{1-t}u_1(x_1-\tau e_i)^t -
2u_0(x_0)^{1-t}u_1(x_1)^t\Big].
  \end{align*} The right hand side is $f''(0)$ for $f(\tau) := u_0(x_0+\tau
e_i)^{1-t}u_1(x_1+\tau e_i)^t$. It's routine to compute
  \begin{align} \nonumber f''(0) &=
\bar{u}(z)\left[(1-t)\frac{D_{ii}u_0}{u_0}(x_0) + t \frac{D_{ii}u_1}{u_1}(x_1) -
t(1-t)\left\vert\frac{D_iu_0}{u_0}(x_0) - \frac{D_{i}u_1}{u_1}(x_1) \right\vert
\right]\\
    \label{eq:to-sub} &= \bar{u}(z)\left[ (1-t)\frac{D_{ii}u_0}{u_0} + t
\frac{D_{ii}u_1}{u_1}\right],
  \end{align} where we've employed \eqref{eq:first-deriv} to obtain
\eqref{eq:to-sub}. Because the $u_i$ are eigenfunctions \eqref{eq:to-sub}
implies
  \begin{align*} \frac{\Delta \bar{u}}{\bar{u}} &\geq (1-t)\frac{\Delta u_0}{u_0} + t \frac{\Delta
u_1}{u_1} \\ &=(1-t)(V(x_0)- \lambda_0) + t (V(x_1)-\lambda_1).
  \end{align*} Subsequently, using the convexity of $V$ to write $V(z) \leq
(1-t)V(x_0)+tV(x_1)$, we obtain at each $z \in \Omega_t$ where $\overline{u}$ is Alexandrov second differentiable
\begin{align*} \bar{u}\Delta \overline{u} &\geq[(1-t)\lambda_0+t\lambda_1]\bar{u}^{2} +
V(z)\bar{u}^2.
\end{align*} The local semiconvexity justifies integration by parts (Lemma
\ref{orgc51eab8}) whereby
  \begin{align*} \int_{\Omega_t} |\nabla \overline{u}|^2 dz &\leq -\int_{\Omega_t} \overline{u} \Delta
\overline{u} dz \\ & \leq \left[(1-t) \lambda_0 + t \lambda_1\right] \int_{\Omega_t} \overline{u}^2 dz
- \int_{\Omega_t}V(z) \overline{u}^2 dz.
  \end{align*} Substituting into the Rayleigh quotient we complete the proof:
  \begin{equation}
    \label{eq:main-fin}
    \lambda_t \leq \frac{\int_{\Omega_t} |\nabla \bar{u}|^2 + V\bar{u}^2dz}{\int_{\Omega_t} \bar{u}^2 dz} \leq
    (1-t)\lambda_0 + t \lambda_1.
  \end{equation}

\end{pf}

The proof of Theorem \ref{thm:log-concave} is an immediate consequence.

\begin{proof}[Proof. (Theorem \ref{thm:log-concave})] Apply the previous proof
with $\Omega_0 = \Omega_1 = \Omega$, so that $u_0 = u_1 = u$ is the first Dirichlet
eigenfunction. The convexity of $\Omega$ implies $\Omega_t = \Omega$. From \eqref{eq:main-fin} $\bar{u}$ minimises the Rayleigh quotient which implies, by uniqueness of principal eigenfunctions up to scaling, that $\bar{u} = C u$ for some $C \in \mathbf{R}$. Evaluate
at a maximum point to obtain $C=1$.  Subsequently $\bar{u} = u$ and the sup-convolution \eqref{eq:sup-conv-def} becomes
  \begin{align*} u(z) = \sup\{u^{1-t}(x_0)u^t(x_1) : z = (1-t)x_0 + tx_1, \> x_i
\in \Omega_i\}.
  \end{align*} Log-concavity, namely that $u$ satisfies
$u\big((1-t)x_0+tx_1\big) \geq u(x_0)^{1-t}u(x_1)^t$, is immediate.
\end{proof}

\begin{rem} Surprisingly, Theorem \ref{org5c5abe2} does not imply convexity of $h(t):= \lambda_t$, which does not hold even for smooth connected open
domains as seen by taking (a smoothing of)
\[ \Omega_0 = \{(x=r\cos \theta, y = r\sin \theta) ; 1 < r < 2 \ , \ \epsilon < \theta < 2\pi\} \subset \mathbf{R}^{2},
\] with $\epsilon>0$ small and $\Omega_1 = B(0,2) \subset \mathbf{R}^{2}$ to obtain discontinuous
$t \mapsto \lambda_t$. This distinction arises because given nonconvex $\Omega_0$, $\Omega_1$ and $t_0,t_1 \in (0,1)$ it is not true that
\[ \{(1-\tau)\Omega_{t_0}+\tau \Omega_{t_1} \}_{\tau \in [0,1]} \subset \{\Omega_t\}_{t \in [t_0,t_1]}\]
Or, more simply, for nonconvex $\Omega$ we may have $(1-t)\Omega + t \Omega \neq \Omega$.
\end{rem}
\begin{rem} The homogeneity of Dirichlet eigenvalues for the Laplacian ($V=0$) means Theorem
 \ref{org5c5abe2} implies $\lambda_t^{-1/2} \geq (1-t)\lambda_0^{-1/2}+t \lambda_1^{-1/2}$ \cite[pg. 128]{zbMATH02181983}, which is usually called the ``Brunn--Minkowski inequality for eigenvalues'' .
\end{rem}

\section{Regularity properties of the sup-convolution}
\label{sec:org184251d} In this section we prove the sup-convolution $\bar{u}$  is locally semiconvex and Lipschitz up to the boundary. Whilst these results are likely obvious to those familiar with semiconvex functions, for the sake of a self-contained paper we give complete proofs. We must take some care since for \(t \in (0, 1)\), \(u_{0}^{1-t}\) and \(u_1^{t}\) may not be Lipschitz up to the boundary as \(u_{i}|_{\partial \Omega_i} = 0\). We must also account for the fact \(\bar{u}\) is defined variationally, hence may not enjoy interior regularity.

\subsection{Minkowski Deformation}
\label{sec:orge2d344d} For open sets $\Omega_0,\Omega_1$ the Minkowski deformations \eqref{eq:mink-sum-def} are open sets which respect closures:
\(\overline{\Omega_t} = (1-t) \overline{\Omega_0} + t \overline{\Omega_1}\). The boundary behaviour is not so well respected.
\begin{lem}
  \label{org44df74f} We have the inclusion
  \[ \partial \Omega_t \subseteq \partial [(1-t)\Omega_0] + \partial [t\Omega_1].
  \]
\end{lem}
\begin{pf} Since the Minkowski deformation respects closures, it suffices to
show that
  \begin{align*} (1-t)\Omega_0 + \overline{t\Omega_1} &\subseteq \Omega_t, \\ \overline{(1-t)\Omega_0} +
t\Omega_1 &\subseteq \Omega_t.
  \end{align*}

We will just show the former with the latter being similar.  Let \(z = (1-t)x_0
+ t x_1 \in (1-t)\Omega_0 + \overline{t\Omega_1}\) and let \(U \subseteq (1-t)\Omega_0\) be an open set
containing \(x_0\). Then \((1-t)U + t x_1\) is an open set
containing \(z\) and contained in \(\Omega_t\), hence \(z\) is in the interior of \(\Omega_t\).
\end{pf}

\begin{rem}
The inclusion from Lemma \ref{org44df74f} is in general strict. Let \(\Omega_0 = \Omega_1 = B(0,1)\) and let \(x_0 = (-1, 0) \in \partial
\Omega_0\), \(x_1 = (1, 0) \in \partial \Omega_1\). Then \((1-1/2)x_0 + (1/2)x_1 = (0, 0) \in
\Omega_{1/2}\)
 Note \(\Omega_0\) and \(\Omega_1\) are uniformly
convex. Thus not even uniform convexity can improve the behaviour of boundaries under  deformations.
\end{rem}
\subsection{Admissible And Optimal Points}
\label{sec:org81a28c0}
\begin{defn}
  \label{orgf87f94b} Let \(z \in \Omega_t\) and \((x_0, x_1) \in \overline{\Omega_0} \times \overline{\Omega_1}\). We call \((x_0, x_1)\) \emph{admissible} for \(z\) if \((1-t)x_0 + t x_1 = z\), and \emph{optimal} if in addition equality is attained in the supremum defining the sup-convolution \eqref{eq:sup-conv-def}.
\end{defn}

We denote the set of continuous functions \(\overline{\Omega_i} \to \mathbf{R}\) vanishing on
\(\partial \Omega_i\) by \(C^0_0(\overline{\Omega_i} \to \mathbf{R})\).

\begin{lem}
  \label{org2260c5a} Let \(u_i \in C^0_0(\overline{\Omega_i} \to \mathbf{R})\) satisfy \(u_i > 0\)
on the interior.  Then for any \(z \in \Omega_t\), the set of optimal \((x_0, x_1) \in
\overline{\Omega_0} \times \overline{\Omega}_1\) is non-empty, and contained in the interior
\(\Omega_0 \times \Omega_1\).
\end{lem}

\begin{pf} The supremum is attained on \(\overline{\Omega_0} \times \overline{\Omega_1}\) by
compactness and continuity, so the set of optimal points is non-empty.
For interior \(z \in \Omega_t\), there exists admissible
\((x_0, x_1)\) with \(x_i\) interior.  Then since \(u_i(x_i) > 0\) we have that
\(\overline{u}(z) \geq u_0(x_0)^{1-t} u_1(x_1)^t > 0\). For admissible \((x_0,
x_1)\) with \(x_i \in \partial \Omega_i\) for either \(i = 1\) or \(i = 2\) (or both),
\(u_0(x_0)^{1-t}u_1(x_1)^t = 0 < \overline{u}(z)\) hence \((x_0, x_1)\) is not
optimal.
\end{pf}

\subsection{Regularity of sup-convolution}
\label{sec:org166c3b2}

Now we study the regularity of $\overline{u}$ assuming $u_i \in C^2(\Omega_i) \cap C^{1}(\overline{\Omega_i})$ is positive on $\Omega_i$ and $0$ on $\partial \Omega_i$. In our setting this is assured by Assumption \ref{regularity-assumption} which is satisfied when for, example, $V$ is H\"older and $\Omega$ is $C^{1,1}$ since, in this case, standard  results imply $u_i \in C^{2,\alpha}(\Omega_i) \cap C^{1}(\overline{\Omega_i})$ \cite[Lemma 6.16, Theorem 9.15]{MR1814364}.
\subsubsection{Semi-convexity of sup-convolution}
\label{sec:org33bf5d3}

\begin{defn} Let \(\Omega\) be an open domain. A function \(u : \Omega \to \mathbf{R}\) is
  \emph{locally semi-convex} if for every compact, convex set \(K \subseteq \Omega\),
  there exists a \(\lambda = \lambda(K) \in \mathbf{R}\) such that
  \[ x \in K \mapsto u(x) + \frac{\lambda}{2}|x|^2
  \] is convex. We say \(u\) is \emph{(uniformly) semi-convex} if \(\lambda\) may be
chosen independently of \(K\).
\end{defn}

Note that if \(u \in C^2(\Omega)\), then \(u\) is locally semi-convex. If \(\Omega\) is
bounded and \(u \in C^2(\overline{\Omega})\), then \(u\) is
uniformly semi-convex. From the corresponding result for convex
functions, a supremum over a family of semiconvex functions with fixed $\lambda$ is
semiconvex with constant $\lambda$.

\begin{lem}
  \label{org15f8af8}

Let \(u_i \in C^2(\Omega_i)\). Then the sup-convolution \(\overline{u}\) is locally semi-convex.
\end{lem}
\begin{proof} Fix any convex $K \subset\joinrel\subset \Omega_t$ and let $\delta = \inf_{z \in
K}\overline{u}(z) > 0$. Set $M = \max{(\Vert u_0 \Vert_{L^\infty(\Omega_0)},\Vert u_1
\Vert_{L^\infty(\Omega_1)})}$ and $\tilde{\Omega}_0 = \{x \in \Omega_0 : u_{0}(x) \geq (\delta/M)^{\frac{1}{1-t}}\}$,
$\tilde{\Omega}_1= \{x \in \Omega_1 : u_{1}(x) \geq (\delta/M)^{\frac{1}{t}}\}$. If $x_0 \in \Omega_0 \setminus \tilde{\Omega}_0$ or
$x_1 \in \Omega_1 \setminus \tilde{\Omega}_1$ then $u_0(x_{0})^{1-t}u_1(x_{1})^{t} < \delta$, which implies
for all $z \in K$
  \[\overline{u}(z) = \sup\{u_0^{1-t}(x_0)u_1^t(x_1) : z = (1-t)x_0 + tx_1, \>
x_i \in \tilde{\Omega_i}\}.  \]

  Let $\tilde{u}_1$ denote the extension of $u_1$ as $(\delta/M)^{1/t}$ on
$\tilde{\Omega}_1^c$. Then with \((x_0, x_1)\) admissible for
\(z\) we have \(x_1 = \frac{t-1}{t} x_0 + \frac{z}{t}\), hence we may write
  \[ \overline{u}(z) =
\sup\left\{u_0^{1-t}(x_0)\tilde{u_1}^t\left(\frac{z}{t}-\frac{1-t}{t}x_0 \right)
: x_0 \in \tilde{\Omega}_0\right\}.\] As a function of \(z\), the term in the sup is semi-convex with constant \(\lambda\) independent of \(x_{0}\). Having expressed $\bar{u}|_{K}$ as a supremum of semiconvex functions, $\bar{u}|_{K}$ is semiconvex.
\end{proof}

\begin{cor}
  \label{org60a33f1} Let\(u_i \in C^2(\Omega_i)\). Then the sup-convolution \(\overline{u}\) is locally Lipschitz
and Alexandrov twice differentiable almost everywhere. That is, there exists
measurable functions \(\nabla \overline{u} : \Omega_t \to \mathbf{R}^n\) and \(D^2 \overline{u} : \Omega_t
\to S_n\) (where \(S_n\) denotes \(n\times n\) nonnegative definite matrices) such that for almost
every \(z_0 \in \Omega_t\)
  \[ \overline{u}(z) = \overline{u}(z_0) + \langle \nabla \overline{u} (z_0), z-z_0\rangle +
\frac{1}{2} (z-z_0)^T D^2 \overline{u} (z_0) (z-z_0) + o (|z-z_0|^2)
  \] as \(z \to z_0\).
\end{cor}

This is standard for convex and hence semi-convex functions. See \cite[Section
6.4]{zbMATH08010281}. It is straightforward to compute the gradient of $\bar{u}$
at points of differentiability.

\begin{cor}\label{eq:deriv-calc} Let $z \in \Omega_t$ and let $(x_0,x_1)$ be optimal
for $z$. If $\overline{u}$ is differentiable at $z$, then
\[ \frac{\nabla \overline{u}}{\overline{u}}(z) = (1-t)\frac{\nabla u_0}{u_0}(x_0) + t
\frac{\nabla u_1}{u_1}(x_1) = \frac{\nabla u_0}{u_0}(x_0) = \frac{\nabla u_1}{u_1}(x_1).
\]
\end{cor}
\begin{proof} The latter two equalities follow from the first and
\eqref{eq:first-deriv}. Next, if $z$ is a point of differentiability, then
optimality of $(x_0,x_1)$ implies
  \begin{align*} D_{i}\overline{u}(z) &= \lim_{h \rightarrow
0}\frac{\overline{u}(z+he_i)-\overline{u}(z)}{h}\\ &\geq \lim_{h \rightarrow
0}h^{-1}\left[u_0(x_0+he_i)^{1-t}u_1(x_1+he_i)^t-u_0(x_0)^{1-t}u_1(x_1)^t\right]\\
&= \overline{u}(z)\left[(1-t)\frac{Du_0}{u_0}(x_0) + t
\frac{Du_1}{u_1}(x_1)\right].
  \end{align*} The opposite inequality uses $ D_{i}\overline{u} = \lim_{h \rightarrow0}h^{-1}\big(\overline{u}(z)-\overline{u}(z-he_i)\big)$ and a similar argument.
\end{proof}
From this formula for the derivatives it's straightforward to deduce better
regularity for $\bar{u}$.
\subsubsection{Sobolev Regularity}
\label{sec:orgf5a19af}
\begin{lem}
  \label{org7214da8} Let \(u_i \in C^2(\Omega_i) \cap C^1(\overline{\Omega_i})\). Then the
sup-convolution is in \(W^{1,p}_0\) for every \(p \geq 1\).
\end{lem}

Here, to say \(u \in C^1(\overline{\Omega})\) is to say that there is a $C^1$ extension of $u$ to an open neighbourhood of $\Omega$.

\begin{pf} Since \(\Omega_t\) is bounded we have the following inclusions
  \[ \operatorname{Lip} \subseteq W^{1,\infty} \subseteq \dots W^{1,p} \subseteq \dots W^{1,1}.
  \]

From Lemma \ref{org60a33f1}, \(\overline{u}\) is locally Lipschitz. We just need
to provide a uniform Lipschitz constant to verify that \(\overline{u} \in
W^{1,p}\) for every \(p \geq 1\). Then since \(\overline{u}\) is continuous and
\(\overline{u} \equiv 0\) on the boundary \(\partial \Omega_t\) we will have \(\overline{u} \in
W^{1,p}_0\).

Since \(u_i\) are $C^1$ up to the boundary, \(|\nabla u_i|\) is uniformly bounded. Since \(\overline{u}\) is locally Lipschitz, it is differentiable almost
everywhere by Rademacher's theorem. At points where \(\overline{u}\) is
differentiable, Corollary \ref{eq:deriv-calc} ensures that
  \[ \frac{\nabla \overline{u}(z)}{u_0^{1-t}(x_0)u_1^t(x_1)} =\frac{\nabla
\overline{u}}{\overline{u}}(z) = \frac{\nabla u_0}{u_0}(x_0) = \frac{\nabla u_1}{u_1}
(x_1).
  \] In the case that \(u_1(x_1) \leq u_0(x_0)\) we have
  \[ |\nabla \overline{u}(z)| = \left|\frac{u_1^t(x_1)}{u_0^t(x_0)}\nabla u_0(x_0)\right|
\leq \left|\nabla u_0(x_0)\right| \leq \left\Vert \nabla u_0\right\Vert_{L^{\infty}}.
  \] Similarly, if \(u_0(x_0) \leq u_1(x_1)\) then $\left|\nabla \overline{u}(z)\right| \leq \left\|\nabla u_1\right\|_{L^{\infty}}$. Thus \(\overline{u}\) is locally Lipschitz with almost-everywhere uniformly
bounded gradient, hence is Lipschitz with Lipschitz constant bounded by \(\max
\{\|\nabla u_0\|_{L^{\infty}}, \|\nabla u_1\|_{L^{\infty}}\}\).
\end{pf}
\subsubsection{BV Gradient}
\label{sec:org919b310}

Lemma \ref{org7214da8} gives first order regularity for \(\overline{u}\). However in general \(\overline{u}\) is not in \(W^{2,1}\). Nevertheless, semiconvexity
gives an integration by parts inequality. The following lemmas, the first of which is well-known but difficult to find in the literature, complete the justification for the proof of Theorem \ref{org5c5abe2}.

\begin{lem}\label{lem:lim-mollified-2deriv}
  Let $\overline{u}:\Omega \rightarrow \mathbf{R}$ be any semiconvex function. For $\epsilon > 0$ set $\overline{u}_\epsilon = \eta_\epsilon \ast \overline{u}$ where $\eta_\epsilon$ is the standard mollifier. If $u$ has Alexandrov second derivatives at $x \in \Omega$, then
  \begin{equation}
    \label{eq:lim-mollified-2deriv}
    \lim_{\epsilon \rightarrow 0} D^2\overline{u}_\epsilon(x) = D^2\overline{u}(x).
  \end{equation}
\end{lem}
\begin{proof}
Using Alexandrov second differentiability
  \begin{align*}
    D_{ij}\overline{u}_{\epsilon}(x) &= \int_{\mathbf{R}^n}D_{ij}\eta_\epsilon(x-y)\overline{u}(y) \, dy\\
                              &= \int_{\mathbf{R}^n}D_{ij}\eta_\epsilon(x-y)\Big[
                                \overline{u}(x) + D\overline{u}(x)\cdot(y-x)\\
    &\quad \quad +\frac{1}{2}D_{kl}\overline{u}(x)(y-x)_{k}(y-x)_{l} + o(|x-y|^2)
                                \Big] \, dy.
  \end{align*}
Note the zeroth and first order terms (i.e.,  $\overline{u}(x)$ and $D\overline{u}(x)\cdot(y-x)$) vanish after integrating by parts on $y$. As do terms $D_{kl}\overline{u}(x)(y-x)_{k}(y-x)_{l}$ for which $\{k,l\} \neq \{i,j\}$. Thus, after integrating by parts we're left with
  \[     D_{ij}\overline{u}_{\epsilon}(x)= D_{ij}\overline{u}(x)\int_{\mathbf{R}^n}\eta_\epsilon(x-y)\, dy + o(\epsilon^2)\int_{\mathbf{R}^{n}} D_{ij}\eta_\epsilon(x-y)\, dy.\]
  The first integral equals 1. For the second, $| \int_{\mathbf{R}^{n}} D_{ij}\eta_\epsilon(x-y)| \leq C\epsilon^{-2}$ allows us to take $\epsilon \rightarrow0$ via the dominated convergence theorem and obtain \eqref{eq:lim-mollified-2deriv}.
\end{proof}

\begin{lem}
  \label{orgc51eab8} Let \(\overline{u} : \in C^0_0(\overline{\Omega})\) be locally semi-convex and globally Lipschitz. Then \(\overline{u}\) satisfies the following integration by parts inequality
  \begin{align}
    \label{eq:ip-ineq} \int_{\Omega_t} |\nabla \overline{u}(x)|^2 dx &\leq -\int_{\Omega_t}
\overline{u}(x) \Delta \overline{u}(x) \, dx,
    \end{align} where $\Delta \overline{u}(x) = \operatorname{Tr} D^{2}\overline{u}(x)$.
\end{lem}
\begin{proof} Fix $\tilde{\Omega}_t \subset\joinrel\subset \Omega_t$ with \(C^{1}\) boundary and $\lambda$ such that $\overline{u}+\lambda|x|^2/2$ is convex on every convex subset of $\tilde{\Omega}_t$. Let $\overline{u}_{\epsilon} = \eta_\epsilon \ast \overline{u}$ where $\eta_\epsilon$ is the standard mollifier. Because the mollification of a convex function is a convex function \cite[Proof of Theorem 6.7]{zbMATH08010281}, and this is a local statement,
  \begin{equation}
    \label{eq:moll-2deriv-bound-below}
    D^2u_\epsilon \geq -\lambda I.
  \end{equation}

  Integrating by parts,
  \begin{align}
    \label{eq:mollified-ip} \int_{\tilde{\Omega}_t}|\nabla \overline{u}_{\epsilon}|^2 \, dx = \int_{\partial
\tilde{\Omega}_t}\overline{u}_{\epsilon} \nabla\overline{u}_{\epsilon} \cdot \mathbf{n} \, dS -
\int_{\tilde{\Omega}_t}\overline{u}_{\epsilon} \Delta \overline{u}_{\epsilon} \, dx.
  \end{align} We have
  \begin{align}
    \label{eq:1st-deriv-ests} \left|\int_{\partial \tilde{\Omega}_t}\overline{u}_{\epsilon}
\nabla\overline{u}_{\epsilon} \cdot \mathbf{n} \, dS\right| \leq C \sup_{\partial
\tilde{\Omega}_t}|\overline{u}_{\epsilon}| \mathcal{H}^{n-1}(\partial \tilde{\Omega}_t),
  \end{align} where $C$ is the Lipschitz constant of $\overline{u}$ (preserved
under mollification). Moreover by \eqref{eq:moll-2deriv-bound-below}, $\overline{u}_\epsilon \Delta \overline{u}_\epsilon $ is bounded below (independently of $\epsilon$) and, by Lemma \ref{lem:lim-mollified-2deriv}, converges pointwise almost everywhere to $\overline{u}\Delta\overline{u}$. Thus, Fatou's lemma implies
  \begin{align} \label{eq:2nd-deriv-convergence} \liminf_{\epsilon \rightarrow 0} \int_{\tilde{\Omega}_t}\overline{u}_{\epsilon} \Delta\overline{u}_{\epsilon} \, dx & \geq \int_{\tilde{\Omega}_t}\overline{u} \Delta \overline{u} \, dx .
  \end{align}
Substituting \eqref{eq:1st-deriv-ests} and \eqref{eq:2nd-deriv-convergence} into \eqref{eq:mollified-ip} and sending $\epsilon \rightarrow 0$ implies
  \begin{align*} \int_{\tilde{\Omega}_t}|\nabla\overline{u}|^2 \, dx &\leq C \sup_{\partial
\tilde{\Omega}_t}|\overline{u}| \mathcal{H}^{n-1}(\partial \tilde{\Omega}_t) - \int_{\tilde{\Omega}_t}\overline{u}
\Delta \overline{u} \, dx.
  \end{align*} Taking a sequence of domains $\tilde{\Omega}_t$ converging to $\Omega_t$
gives \eqref{eq:ip-ineq} since \(\overline{u}|_{\partial \Omega_t} \equiv 0\).
\end{proof}

\printbibliography
\end{document}